\documentclass[a4paper,11pt,twoside]{article}
\usepackage{latexsym}
\usepackage{a4wide}
\usepackage{amscd}
\usepackage{graphics}
\usepackage{graphicx}
\usepackage[english]{varioref}
\usepackage{amsmath}
\usepackage{amssymb}
\usepackage{mathrsfs}
\usepackage{amsthm}
\usepackage{amssymb,tikz}
\usepackage{verbatim}

\usepackage{bbm}
\usepackage{color}
\usepackage{accents}
\usepackage{enumerate}
\usepackage[utf8x]{inputenc}
\usepackage[T1]{fontenc}
\usepackage[english]{babel}
\usepackage{amsfonts, amscd, amsmath, amssymb, amsthm, mathrsfs, mathtools}
\usepackage{braket}
\input xy
\xyoption{all}
\usepackage[utf8x]{inputenc}
\usepackage[T1]{fontenc}
\usepackage[english]{babel}
\usepackage{amsfonts, amscd, amsmath, amssymb, amsthm, mathrsfs, mathtools}
\usepackage{graphicx}
\usepackage{subfig}
\usepackage{braket}
\usepackage[left=2.5cm,right=2.5cm,top=2.5cm,bottom=2.5cm]{geometry}
\usepackage{fancyhdr}
\theoremstyle{definition}

\newtheorem{remark}{Remark}[section]

\newtheorem{esempio}{Example}[section]

\theoremstyle{plain}
\newtheorem{definizione}{Definition}[section]
\newtheorem{teorema}{Theorem}[section]
\newtheorem{proposizione}{Proposition}[section]
\newtheorem{lemma}{Lemma}[section]
\newtheorem{corollario}{Corollary}[section]

\newcommand{\numberset}{\mathbb}

\newcommand{\R}{\numberset{R}}
\newcommand{\N}{\numberset{N}}
\newcommand{\Z}{\numberset{Z}}

\usepackage{bookmark,hyperref}

\DeclarePairedDelimiter{\abs}{\lvert}{\rvert}
\DeclarePairedDelimiter{\norma}{\lVert}{\rVert}

\makeatletter
\let\oldabs\abs
\def\abs{\@ifstar{\oldabs}{\oldabs*}}
\let\oldnorma\norma
\def\norma{\@ifstar{\oldnorma}{\oldnorma*}}

\title{Existence of Lipschitz continuous Lyapunov functions \\ strict outside the strong chain recurrent set}

\author{\hspace{-1cm}{OLGA BERNARDI${\,
}^{1}$ \quad ANNA FLORIO${\,
}^{2}$                  }
\vspace{0.3cm}
\\
\hspace{-1cm}${\ }^{1}$ Dipartimento di Matematica ``Tullio Levi-Civita'',
Universit\`a di Padova,\\
\hspace{-1cm}Via Trieste, 63 - 35121 Padova, Italy
\\ 
\\
\hspace{-1cm}${\ }^{2}$ Laboratoire de Mathématiques d'Avignon, Avignon Université, \\
\hspace{-1cm} 84018 Avignon, France \\}

\date{ }

\begin{document}

\maketitle

\begin{abstract}
\noindent The aim of this paper is to study in detail the relations between strong chain recurrence for flows and Lyapunov functions. For a continuous flow on a compact metric space, uniformly Lipschitz continuous on the compact subsets of the time, we first make explicit a Lipschitz continuous Lyapunov function strict --that is strictly decreasing-- outside the strong chain recurrent set of the flow. This construction extends to flows some recent advances of Fathi and Pageault in the case of homeomorphisms; moreover, it improves Conley's result about the existence of a continuous Lyapunov function strictly decreasing outside the chain recurrent set of a continuous flow. We then present two consequences of this theorem. From one hand, we characterize the strong chain recurrent set in terms of Lipschitz continuous Lyapunov functions. From the other hand, in the case of a flow induced by a vector field, we establish a sufficient condition for the existence of a $C^{1,1}$ strict Lyapunov function and we also discuss various examples. Moreover, for general continuos flows, we show that the strong chain recurrent set has only one strong chain transitive component if and only if the only Lipschitz continuous Lyapunov functions are the constants. Finally, we provide a necessary and sufficient condition to guarantee that the strong chain recurrent set and the chain recurrent one coincide.
\end{abstract}

\section{Introduction}
Let $\phi = \{\phi_t\}_{t \in \R}$ be a continuous flow on a compact metric space $(X,d)$. In the paper \cite{east41}[Page 96], Robert Easton introduced the notion of strong chain recurrence. \\
~\newline
DEFINITION. (Strong chain recurrence)
\begin{itemize}
\item[$(i)$] 
\emph{\indent Given $x,y\in X$, $\varepsilon>0$ and $T > 0$, a strong $(\varepsilon,T)$-chain from $x$ to $y$ is a finite sequence $(x_i,t_i)_{i=1,\dots,n}\subset X\times\R$ such that $t_i\geq T$ for all $i$, $x_1=x$ and setting $x_{n+1}=y$ we have
\begin{equation} \label{prima intro ita}
\sum_{i=1}^nd(\phi_{t_i}(x_i),x_{i+1}) < \varepsilon.
\end{equation}
\item[$(ii)$] 
A point $x\in X$ is said to be strong chain recurrent if for all $\varepsilon>0$ and $T > 0$ there exists a strong $(\varepsilon,T)$-chain from $x$ to $x$. The set of strong chain recurrent points is denoted by $\mathcal{SCR}(\phi)$.
\item[$(iii)$] 
The points $x,y\in\mathcal{SCR}(\phi)$ belong to the same strong chain transitive component of $\mathcal{SCR}(\phi)$ if for any $\varepsilon>0$ and $T > 0$ there exist a strong $(\varepsilon,T)$-chain from $x$ to $y$ and a strong $(\varepsilon,T)$-chain from $y$ to $x$.} 
\end{itemize}
\noindent This notion sharpens the one of chain recurrence, in which it is only required
\begin{equation} \label{seconda intro ita}
d(\phi_{t_i}(x_i),x_{i+1}) < \varepsilon
\end{equation} for all $i$. Chain recurrent sets are both easily seen to be closed and invariant, see e.g. \cite{akin46}{[Pages 12-13, 71 and 109-110]}, \cite{conARTICOLO}{[Theorem 3.3B]}, \cite{zz00}{[Page 248]} and \cite{Bernardiflorio}{[Proposition 3.1]}. 
Furthermore, strong chain recurrent set (strictly) contains the set of non-wandering points and it is (strictly) contained in the chain recurrent set, that is $\mathcal{NW}(\phi) \subseteq \mathcal{SCR}(\phi) \subseteq \mathcal{CR}(\phi)$, see for example \cite{zz00}{[Theorem 2.4]} and \cite{conARTICOLO}[Theorem 3.3B]. \\
\indent The study of the intimate relations between chain recurrent sets and Lyapunov functions comes from the celebrated paper \cite{conl33} by Charles Conley and has had recent important advances by Albert Fathi and Pierre Pageault in \cite{fathi2015} and \cite{pageault}. Before recalling these results, we need to premise the notions of Lyapunov function, neutral set and first integral. \\
~\newline
DEFINITION. (Lyapunov function, neutral set and first integral) \\
\emph{A function $f:X\rightarrow\R$ is a Lyapunov function for $\phi$ if $f\circ\phi_t\leq f$ for every $t\geq 0$.} \emph{The neutral set of a Lyapunov function $f$ is 
$$\mathcal{N}(f) := \{x \in X: \ \exists t > 0 \text{ such that } f(\phi_t(x)) = f(x)\}.$$}
\emph{A function $f:X\rightarrow\R$ is a first integral for $\phi$ if $f\circ\phi_t = f$ for every $t\in \R$.} \\
~\newline
\noindent In particular, $f$ is a first integral if and only if $\mathcal{N}(f) = X$. We refer to Lemma 1.5 in \cite{abbbercar} for a characterization of Lyapunov functions and first integrals in the case of a flow induced by a locally Lipschitz continuous vector field. \\
\indent In the seminal paper \cite{conl33}, Conley described the structure of the chain recurrent set in terms of attractors and their ``complementary repellers'' and --as an outcome-- he proved the theorem below. Some authors refer to this result as the ``fundamental theorem of dynamical systems'', see e.g. \cite{norton} and \cite{franksarxiv} for an instructive treatment of the matter. \\
~\newline
THEOREM. (Conley, 1978) \\
\emph{Let $\phi: X \times \R \to X$ be a continuous flow on a compact metric space. Then there exists a continuous Lyapunov function $u:X \to \R$ for $\phi$ such that
\begin{itemize}
\item[$(i)$] $\mathcal{N}(u) = \mathcal{CR}(\phi)$.
\item[$(ii)$] If $x,y \in \mathcal{CR}(\phi)$, then $u(x) = u(y)$ if and only if $x$ and $y$ belong to the same chain transitive component of $\mathcal{CR}(\phi)$.
\item[$(iii)$] $u(\mathcal{CR}(\phi))$ is a compact nowhere dense subset of $\R$. 
\end{itemize}}
\noindent Conley called $u$ complete Lyapunov function for $\phi$, see Definition \ref{completa} of Section \ref{cinque}. In particular, point $(i)$ of the previous theorem is an either/or statement about what can happen: parts of the space are either chain recurrent or ``gradient-like''. In the same year --see \cite{east41}[Proposition 3]-- Easton connected the notion of strong chain recurrence to the property for the corresponding Lipschitz first integrals to be constants. Easton's contribution represented the first step towards the study of the relations between strong chain recurrence and Lyapunov functions. \\
\indent For a homeomorphism $g$ on a compact metric space, Fathi and Pageault in \cite{fathi2015} and \cite{pageault} presented a new variational point of view to face the study of recurrent sets and Lyapunov functions. Their techniques are very different from Conley's original ones and are inspired by Fathi's work in weak KAM theory, see \cite{fathiWeakKAM}. 
Indeed, thanks to the arbitrariness of the parameter $\varepsilon > 0$ involved in both the definitions of recurrence, Fathi and Pageault equivalently described recurrent points as minima of appropriate functionals defined on the space of finite sequences of points. In particular, bearing in mind formulae (\ref{prima intro ita}) and (\ref{seconda intro ita}), for strong chain recurrent points the functional is the sum of the amplitudes of the jumps; for chain recurrent points, the functional is the maximum of the amplitudes of the jumps. For a discrete dynamical system, if $u$ is a Lyapunov function for $g$, that is $u \circ g \le u$, the neutral set of $u$ is
$$\mathcal{N}(u) := \{x \in X: \ u(g(x)) = u(x) \}.$$
We remind two of their fundamental results. On one hand, they provided a new proof of the above Conley's theorem in the framework of discrete-time dynamical systems. On the other hand, they proved the next \\
~\newline
THEOREM. (Fathi \& Pageault, 2015). \\
\emph{Let $g: X \to X$ be a homeomorphism on a compact metric space. Then 
\begin{itemize} 
\item[$(i)$] There exists a Lipschitz continuous Lyapunov function $u$ for $g$ such that $\mathcal{N}(u) = \mathcal{SCR}(g)$.
\item[$(ii)$] $\mathcal{SCR}(g) = \bigcap \mathcal{N}(f)$, where the intersection is taken over all Lipschitz continuous Lyapunov functions $f$ for $g$.
\item[$(iii)$] $\mathcal{SCR}(g)$ has a unique strong chain transitive component if and only if the only Lipschitz continuous Lyapunov functions for $g$ are the constants.
\end{itemize}}
\noindent We refer respectively to Theorem 4.4, Corollary 4.5 and Theorem 4.8 in \cite{pageault} for the proofs of the points $(i)$, $(ii)$ and $(iii)$ above. More recently, an adaptation of Fathi and Pageault's techniques has led to the proof of point $(ii)$ of the previous theorem also for a flow which is Lipschitz continuous for every $t \ge 0$, uniformly for $t$ on compact subsets of $[0,+\infty)$, see \cite{abbbercar}[Theorem 2.2]. However, in \cite{abbbercar} it is not constructed a single Lyapunov function whose neutral set coincides with the strong chain recurrent set. We finally underline that, more recently, Fathi and Pageault in \cite{TAMS} give a criterion for the approximation of a Lyapunov function by a smooth one and also show that it is possible to obtain a smooth version of Conley's fundamental theorem for flows. \\
\indent This paper intends to examine in depth the relations between strong chain recurrence and Lyapunov functions in the case of flows. First,  we prove this improvement of point $(i)$ of Conley's original result: \\
~\newline
THEOREM 1. \textit{Let $\phi:X\times\R\rightarrow X$ be a continuous flow on a compact metric space $(X,d)$, uniformly Lipschitz continuous on the compact subsets of $[0,+\infty)$. Then there exists a Lipschitz continuous Lyapunov function $u:X\rightarrow\R$ for $\phi$ such that}
\begin{equation*}
\mathcal{N}(u) =\mathcal{SCR}(\phi).
\end{equation*}
The proof of this theorem combines together variational and dynamical methods. In particular, we start by constructing --by accurately adapting Fathi and Pageault's techniques-- a function which strictly decreases along the flow out of a closed set containing $\mathcal{SCR}(\phi)$ and definitively (that is, for $t \ge T$) in time. We refer to Proposition \ref{proprieta u_T} for the precise statement. Thereafter, by using some ideas coming from the original proof of Conley's theorem, we modify this function in order to obtain the desired Lyapunov function, that is strictly decreasing outside $\mathcal{SCR}(\phi)$ for any $t > 0$. See Lemma \ref{proprieta utilde}, Lemma \ref{proprieta baru} and the consequent Theorem \ref{main thm}. \\
\noindent We proceed by discussing the main consequences of the above theorem. From one hand, we give an alternative --i.e. constructive-- proof of Theorem 2.2 in \cite{abbbercar}: \\
~\newline
\noindent COROLLARY 1. \emph{Let $\phi:X\times\R\rightarrow X$ be a continuous flow on a compact metric space $(X,d)$, uniformly Lipschitz continuous on the compact subsets of $[0,+\infty)$. Then
\begin{equation*}
\mathcal{SCR}(\phi)=\bigcap\mathcal{N}(f).
\end{equation*}
where $f$ is any Lipschitz continuous Lyapunov function for $\phi$.} \\
~\newline
From the other hand, we discuss a sufficient condition in order to establish the existence of a $C^{1,1}$ strict Lyapunov function for a flow induced by a vector field: \\
~\newline
COROLLARY 2. \textit{Let $(M,g)$ be a $C^{\infty}$ closed connected Riemannian manifold. Let $V: M \to TM$ be a $C^k$ vector field, $k \ge 2$, inducing the flow $\phi$. If $
\mathcal{SCR}(\phi)=\mathcal{CR}(\phi)$ then there exists a ${\cal{C}}^{1,1}$ Lyapunov function $u: M \to \R$ for $\phi$ such that} $
\mathcal{N}(u) =\mathcal{SCR}(\phi).
$ \\
~\newline
\noindent See respectively Corollary \ref{altra dim} and Corollary \ref{corollario mane}. In particular, the proof of the second corollary uses existence and regularity results for sub-solutions of the so-called Mañé Hamiltonians. We then pass to discuss some cases --Examples \ref{41} and \ref{es2}-- where the hypothesis ${\mathcal{SCR}}(\phi) = {\mathcal{CR}}(\phi)$ of this corollary is not satisfied and any sub-solutions of weak KAM theory provide first integrals, while the Lyapunov function previously constructed is not a first integral. \\
We also remark that the converse implication of this corollary does not in general hold true. To be thorough, in Definition \ref{completa} we remind the milder notion of pseudo-complete Lyapunov function and in Proposition \ref{ultima car} we prove a necessary and sufficient condition for the strong chain recurrent and the chain recurrent sets to be equal: \\
~\newline
PROPOSITION. {\emph{Let $\phi:X\times\R\rightarrow X$ be a continuous flow on a compact metric space $(X,d)$. Then, $\mathcal{SCR}(\phi) = \mathcal{CR}(\phi)$ if and only if there exists a pseudo-complete Lyapunov function $u: X \to \R$ for $\phi$ such that $u(\mathcal{SCR}(\phi))$ is totally disconnected.}} \\
~\newline
\noindent We finally analyze the relations between Lipschitz continuous Lyapunov functions and the strong chain transitive components of $\mathcal{SCR}(\phi)$. As a final outcome --see Theorem \ref{FP for flows}-- we show that point $(iii)$ of Fathi and Pageault's theorem still holds in the case of a continuous flow: \\
~\newline
THEOREM 2. \emph{Let $\phi:X\times\R\rightarrow X$ be a continuous flow on a compact metric space $(X,d)$. $\mathcal{SCR}(\phi)$ has a unique strong chain transitive component if and only if the only Lipschitz continuous Lyapunov functions for $\phi$ are the constants.} \\
~\newline
\noindent \textbf{Acknowledgements.} O.~Bernardi has been supported by the project CPDA149421/14 of the University of Padova. O.~Bernardi and A.~Florio acknowledge the support of G.N.F.M. The authors also thank the anonymous referee for the careful reading of the    manuscript and his/her many insightful comments and suggestions. 

\section{The function $L_T$: definition and properties}\label{section 2}
\noindent Let $\phi:X \times \R \to X$ be a continuous flow on a compact metric space $(X,d)$. The next definitions are the continuous-time versions of the ones introduced by Fathi and Pageault for homeomorphisms, see \cite{pageault}[Chapter 2, Section 3] and \cite{fathi2015}[Section 2.1]. We also underline that their settings have recently been extended by Ethan Akin and Jim Wiseman in  \cite{AW2017} both to relations and to uniform spaces. \\
\indent Let $x,y\in X$. For any $T > 0$, we indicate by $C_T(x,y)$ the set of chains $C = (x_i,t_i)_{i=1,\dots,n}\subset X\times \R$ from $x_1 =x$ to $x_{n+1} = y$ such that $t_i\geq T$ for all $i$. The cost of going from $x$ to $y$ through a chain $C \in C_T(x,y)$ is given by
$$l_T(C) := \sum_{i=1}^n d(\phi_{t_i}(x_i),x_{i+1}).$$
Moreover, for any $T > 0$, we define the non-negative function
\begin{equation}\label{def L_T}
\begin{split}
L_T&:X\times X\rightarrow [0,+\infty) \\
L_T(x,y)&:=\inf\{ l_T(C):\ C\in C_T(x,y) \}.
\end{split}
\end{equation}
\noindent In the next proposition, we summarize some useful facts about the function $L_T$. We refer to \cite{pageault}[Proposition 3.1] and \cite{fathi2015}[Proposition 2.1] for analogous results in the case of a homeomorphism.
\begin{proposizione}\label{proprieta L_T} 
Let $\phi:X\times\R\rightarrow X$ be a continuous flow on a compact metric space $(X,d)$. For any fixed $T > 0$, the following properties hold:
\begin{itemize}
\item[$(i)$] For any $x,y,z\in X$,
$$
L_T(x,y)\leq L_T(x,z)+L_T(z,y).
$$
\item[$(ii)$] For any $x\in X$,
$$
L_T(x,\phi_t(x))=0 \qquad \forall t\geq T.
$$
\item[$(iii)$] For any $x,y,z\in X$,
$$
\abs{L_T(x,y)-L_T(x,z)}\leq d(y,z).
$$
As a consequence, for any fixed $x \in X$, the function $z \mapsto L_T(x,z)$ is $1$-Lipschitz continuous. 
\item[$(iv)$] For any fixed $x \in X$, the function $z \mapsto L_T(z,x)$ is upper semicontinuous.
\end{itemize}
\end{proposizione}
\begin{proof} For points $(i),$ $(ii)$ and $(iii)$ we refer to Lemma 2.4 in \cite{abbbercar}.\\
$(iv)$ 
By definition of $L_T$, for any $z\in X$ and $\varepsilon>0$, from the definition of infimum, there exists a chain $C_{\varepsilon} \in C_T(z,x)$ such that 
$$
l_T(C_{\varepsilon})<L_T(z,x)+\dfrac{\varepsilon}{2}.
$$
Let $t_1=t_1(z)\geq T$ be the first time of the chain $C_{\varepsilon}$. Since $\phi_{t_1}$ is a homeomorphism, for any $\varepsilon>0$ there exists $\delta=\delta(\varepsilon)>0$ so that
$$
d(z,y)<\delta\qquad\Rightarrow\qquad d(\phi_{t_1}(z),\phi_{t_1}(y))<\dfrac{\varepsilon}{2}.
$$
Let us consider the chain $\tilde{C}_{\varepsilon}\in C_T(y,x)$ obtained from $C_{\varepsilon}$ by substituting the first point with $y$. Hence, if $d(z,y)<\delta$, we have that
$$
L_T(y,x)\leq l_T(\tilde{C}_\varepsilon)\leq d(\phi_{t_1}(z),\phi_{t_1}(y))+l_T(C_{\varepsilon})<\varepsilon + L_T(z,x).
$$
\end{proof}
\indent For any $T > 0$, we define the subset $\mathcal{A}_T$ of $X$ as
\begin{equation}\label{def A_T}
\mathcal{A}_T:=\{x\in X:\ L_T(x,x)=0\}.
\end{equation}
Since $C_T(x,y)$ is contained in $C_{T'}(x,y)$ when $T \ge T'$, the function $T \mapsto L_T(x,y)$ is monotonically increasing for any pair $(x,y) \in X \times X$ and $\mathcal{A}_{T}\subseteq\mathcal{A}_{T'}$. We notice that
\begin{equation}\label{caratterizzazione scr at}
\mathcal{SCR}(\phi) = \bigcap_{T > 0}\mathcal{A}_T.
\end{equation}
The next proposition characterizes the points of the strong chain recurrent set in terms of the functions $L_T$. 
\begin{proposizione}\label{invarianza}
Let $\phi:X\times\R\rightarrow X$ be a continuous flow on a compact metric space $(X,d)$.
\begin{itemize}
\item[$(i)$] If $x\in\mathcal{SCR}(\phi)$, then $L_T(\phi_t(x),x) = L_T(x,\phi_t(x))= 0$ for any $T > 0$ and $t \in \R$.
\item[$(ii)$] If for any $T > 0$ there exists a time $t = t(T) \in \R$ such that  $L_T(\phi_t(x),x) = L_T(x,\phi_t(x))= 0$, then $x\in\mathcal{SCR}(\phi)$.
\end{itemize}
\noindent In particular, if $x  \in \mathcal{SCR}(\phi)$ then $x$ and $\phi_t(x)$ belong to the same strong chain transitive component of $\mathcal{SCR}(\phi)$ for any $t \in \R$.
\end{proposizione}
\begin{proof} $(i)$ Let $x\in\mathcal{SCR}(\phi)$, $T > 0$ and $t\in \R$ be fixed and take $N = N(T,t)\in\N$ such that $t\leq NT$. Since $x \in \mathcal{SCR}(\phi)$, $L_{(N+1)T}(x,x)=0$. This means that for any $\varepsilon>0$ there exists a chain $C_{\varepsilon}=(x_i,t_i)_{i = 1,\ldots, n} \in C_{(N+1)T}(x,x)$ such that $$
l_{(N+1)T}(C_{\varepsilon})\leq\varepsilon.$$
Let consider the following chain $\tilde{C}_{\varepsilon} \in C_T(\phi_t(x),x)$ obtained from $C_{\varepsilon}$:
$$\tilde{C}_{\varepsilon}= \left((\phi_t(x),t_1-t),(x_2,t_2),\dots,(x_n,t_n)\right).$$
Then it clearly holds that
$$
0\leq L_T(\phi_t(x),x)\leq l_T(\tilde{C}_{\varepsilon})=l_{(N+1)T}(C_{\varepsilon})\leq\varepsilon.
$$
From the arbitrariness of $\varepsilon >0$, we deduce that $L_T(\phi_t(x),x)=0$. Since the above argument can be repeated for any fixed $x\in\mathcal{SCR}(\phi)$, $T > 0$ and $t\in \R$, we conclude that if $x\in\mathcal{SCR}(\phi)$ then 
$L_T(\phi_t(x),x)= 0$ for any $T > 0$ and $t \in \R$. \\
\indent For a fixed $x\in\mathcal{SCR}(\phi)$, $T > 0$ and $t\in \R$, we now prove the other equality. If $t\geq T$, the fact that $L_T(x,\phi_t(x))=0$ corresponds exactly to property $(ii)$ of Proposition \ref{proprieta L_T}. Consequently, let $t < T$ and choose $N = N(T,t) \in\N$ such that $T < N |t|$. Moreover, observe that by the invariance of the strong chain recurrent set, $\phi_t(x)\in\mathcal{SCR}(\phi)$. This means that, for any $\varepsilon>0$ there exists a chain $C_{\varepsilon}=(x_i,t_i)_{i = 1,\ldots, n} \in C_{(N + 1)|t|}(\phi_t(x),\phi_t(x))$ such that 
$$l_{(N+1)|t|}(C_{\varepsilon})\leq\varepsilon.$$
\noindent Let us consider the following chain $\tilde{C}_{\varepsilon} \in C_T(x,\phi_t(x))$ obtained from $C_{\varepsilon}$:
$$\tilde{C}_{\varepsilon} = \left((x,t_1+t),(x_2,t_2),\dots,(x_n,t_n)\right),$$
so that
$$
0\leq L_T(x,\phi_t(x))\leq l_T(\tilde{C}_{\varepsilon})=l_{(N+1)|t|}(C_{\varepsilon})\leq\varepsilon.
$$
From the arbitrariness of $\varepsilon > 0$, we conclude that $L_T(x,\phi_t(x))=0$. Since the above argument holds for any fixed $x\in\mathcal{SCR}(\phi)$, $T > 0$ and $t < T$, we conclude that if $x\in\mathcal{SCR}(\phi)$ then 
$L_T(x,\phi_t(x))= 0$ for any $T > 0$ and $t \in \R$. \\
\noindent $(ii)$ Conversely, for a point $x \in X$, let assume that for any $T > 0$ there exists a time $t = t(T) \in \R$ such that
$$
L_T(\phi_t(x),x)=L_T(x,\phi_t(x))=0
$$
Thanks to property $(i)$ of Proposition \ref{proprieta L_T}, we have that
$$
0\leq L_T(x,x)\leq L_T(x,\phi_t(x))+L_T(\phi_t(x),x)=0.
$$
From the arbitrariness of $T > 0$, we deduce that $x\in\mathcal{SCR}(\phi)$.
\end{proof}

\section{Construction of Lipschitz Lyapunov functions definitively strict} \label{quattro}
Through the whole Sections \ref{quattro} and \ref{cinque}, we assume that $\phi: X \times \R \to X$ is  a  continuous  flow on  a  compact  metric  space,  uniformly  Lipschitz  continuous  on every  compact subset  of $[0,+\infty).$ \noindent This means that for any $T > 0$ there exists $M_T > 0$ such that 
\begin{equation} \label{flow ip}
d(\phi_t(x),\phi_t(y)) \le M_T d(x,y) \qquad \forall t \in [0,T].
\end{equation}
\noindent The above Lipschitz regularity assumption is surely satisfied by the flow of a Lipschitz continuous vector field on a compact manifold. Moreover, we precise that we can make a measurable choice of the function $s \mapsto M_s$, for instance
\begin{equation} \label{misurabile}
M_s := \max \left( \sup_{t \in [0,s], \ x \ne y} \frac{d(\phi_t(x), \phi_t(y))}{d(x,y)}, 1 \right).
\end{equation}
\indent Let $\mathcal{K} := \{ x_1, \ldots, x_j, \ldots \}$ be a countable dense subset of $X$; such a set exists since $X$ is a compact metric space. For any $T > 0$, we define the function 
\begin{equation}\label{definizione u_T}
\begin{split} 
u_T&:X\rightarrow\R \\
u_T(x)&:= \sum_{j\in\N} \frac{1}{2^j}L_T(x_j,x).
\end{split}
\end{equation}
We underline that this technique is the same of Fathi and Pageault, see \cite{fathi2015}[Proposition 2.3] and \cite{pageault}[Chapter 2, Proposition 4.1]. Clearly, the function $u_T$ is bounded. Indeed, for any $x \in X$, it holds that
$$
\abs{u_T(x)}\leq \sum_{j\in\N}\dfrac{1}{2^j}L_T(x_j,x)\leq \sum_{j\in\N}\dfrac{1}{2^j} d(\phi_T(x_j), x) \leq 2\ diam(X)
$$
where $diam (X) = \max \{d(x,y): \ x,y \in X\} < +\infty$. In the next proposition we summarize the properties of $u_T$. 
\begin{proposizione}\label{proprieta u_T}
Let $\phi:X\times\R\rightarrow X$ be a continuous flow on a compact metric space $(X,d)$, uniformly Lipschitz continuous on compact subsets of $[0,+\infty)$. For any fixed $T > 0$, the following properties hold:
\begin{itemize}
\item[$(i)$] $u_T$ is $2$-Lipschitz continuous.
\item[$(ii)$] $u_T$ is definitively a Lyapunov function for $\phi$, that is
$$
u_T(\phi_t(x))\leq u_T(x) \qquad \qquad \text{for any } x\in X \text{ and } t \ge T.$$
\item[$(iii)$] $u_T$ is definitively strict outside $\mathcal{A}_{\frac{T}{2}}$, that is
$$
u_T(\phi_t(x))<u_T(x) \qquad \qquad \text{for any } x\in X \setminus \mathcal{A}_{\frac{T}{2}} \text{ and } t \ge T.$$
\end{itemize}
\end{proposizione}
\begin{proof} $(i)$ Thanks to property $(iii)$ of Proposition \ref{proprieta L_T}, for any $x,y\in X$ it holds that
$$
\abs{u_T(y)-u_T(x)}\leq\sum_{j\in\N}\dfrac{1}{2^j}\abs{L_T(x_j,y)-L_T(x_j,x)}\leq 2d(x,y)
$$
$(ii)$ By point $(i)$ of Proposition \ref{proprieta L_T}, for any $x,y\in X$ we have that
$$
u_T(y)-u_T(x) =\sum_{j\in\N}\dfrac{1}{2^j}\left( L_T(x_j,y)-L_T(x_j,x)\right)\leq 2 L_T(x,y)
$$
For $y=\phi_t(x)$ and $t\geq T$, we obtain
$$
u_T(\phi_t(x))-u_T(x)\leq 2 L_T(x,\phi_t(x)) = 0
$$
where --in the last equality-- we use property $(ii)$ of Proposition \ref{proprieta L_T}. \\
$(iii)$ 
\noindent \noindent Arguing by contradiction, let us suppose that $u_T$ is not definitively strict outside $\mathcal{A}_{{\frac{T}{2}}}$. This means that there exist $z \in X \setminus \mathcal{A}_{\frac{T}{2}}$ and $s\geq T$ such that
\begin{equation} \label{nostra ip}
u_T(\phi_{s}(z)) - u_T(z) = \sum_{j\in\N}\dfrac{1}{2^j}\left( L_T(x_j,\phi_{s}(z))-L_T(x_j,z)\right) =0.
\end{equation}
Since from properties $(i)$ and $(ii)$ of Proposition \ref{proprieta L_T}, it holds 
$$
L_T(x_j,\phi_t(x))-L_T(x_j,x)\leq 0 \qquad \qquad \forall x \in X, \ j\in\N \text{ and } t\geq T,$$ 
hypothesis (\ref{nostra ip}) equals to
\begin{equation}\label{hyp contr}
L_T(x_j,\phi_s(z))-L_T(x_j,z)=0\qquad \qquad \forall j\in\N.
\end{equation}
Let $m \in \N$ be fixed. In the sequel we prove that there exists $x_{j(m)} \in \mathcal{K}$ such that
\begin{equation} \label{eccola}
L_T(x_{j(m)},\phi_s(z))<\dfrac{1}{m}.
\end{equation}
Notice that, since $\mathcal{K}$ is a dense sequence in $X$, it is always possible to take $x_{j(m)}\in\mathcal{K}$ such that $d(x_{j(m)},z)<\dfrac{1}{mM_{2T}}$. By hypothesis (\ref{flow ip}), this implies
\begin{equation}\label{oss2}
d(\phi_t(x_{j(m)}),\phi_t(z))<\dfrac{1}{m} \qquad \qquad \forall t\in[0,2T].
\end{equation}
\noindent We now explicitly exhibit a chain $C\in C_T(x_{j(m)},\phi_s(z))$ such that $$l_T(C)<\frac{1}{m}.$$
From one hand, if $s\in[T,2T]$ then
$C = (x_{j(m)},s).$ Indeed, by \eqref{oss2}, we have
$$
l_T(C) = d(\phi_s(x_{j(m)}),\phi_s(z))<\dfrac{1}{m}.
$$
From the other hand, if $s>2T$, let consider the chain $C\in C_T(x_{j(m)},\phi_s(z))$ defined as follows:
$$
C = \left((x_{j(m)},T),(\phi_T(z),s-T)\right).
$$
By (\ref{oss2}) again, it holds 
$$
l_T(C) = d(\phi_T(x_{j(m)}),\phi_T(z))+d(\phi_s(z),\phi_s(z))<\dfrac{1}{m}+0=\dfrac{1}{m}.
$$
Since $L_T(x_{j(m)},\phi_s(z)) \le l_T(C)$ for any $C \in C_T(x_{j(m)},\phi_s(z))$, inequality (\ref{eccola}) immediately follows. Consequently --see (\ref{hyp contr})-- we obtain: 
\begin{equation}
L_T(x_{j(m)},z)<\dfrac{1}{m}.
\end{equation}
By using previous inequality, we are now going to conclude that $L_{\frac{T}{2}}(z,z) = 0$. Given an arbitrary chain $C\in C_T(x_{j(m)},z)$ from $x_{j(m)}$ to $z$:
$$
C = \left((x_{j(m)},t_1),(y_2,t_2),\dots,(y_n,t_n)\right),
$$
let us consider the chain $\tilde{C}\in C_{\frac{T}{2}}(z,z)$ obtained by $C$ and defined as follows:
$$
\tilde{C} =\left( \left(z,\dfrac{T}{2}\right),\left(\phi_{\frac{T}{2}}(x_{j(m)}),t_1-\dfrac{T}{2}\right),\left(y_2,t_2\right),\dots,\left(y_n,t_n\right)\right).
$$
By (\ref{oss2}), we have 
$$
L_{\frac{T}{2}}(z,z)\leq l_{\frac{T}{2}}(\tilde{C}) = d(\phi_{\frac{T}{2}}(z),\phi_{\frac{T}{2}}(x_{j(m)})) + d(\phi_{t_1}(x_{j(m)}),y_2) + \dots+d(\phi_{t_n}(y_n),z)<\dfrac{1}{m} + l_T(C).
$$
Since $L_T(x_{j(m)},z)<\frac{1}{m}$, by taking the infimum over all possible chains in $C_T(x_{j(m)},z)$, we have
$$
L_{\frac{T}{2}}(z,z) \le \dfrac{1}{m}+\dfrac{1}{m}=\dfrac{2}{m}.
$$
From the arbitrariness of $m\in\N$, we obtain
$
L_{\frac{T}{2}}(z,z) = 0.
$ 
Since $z\in X \setminus \mathcal{A}_{\frac{T}{2}}$, previous equality gives the desired contradiction.
\end{proof}

\section{Existence of Lipschitz Lyapunov functions with $\mathcal{SCR}(\phi)$ as neutral set} \label{cinque}
For a continuous flow $\phi: X \times \R \to X$ satisfying hypothesis (\ref{flow ip}), this section is devoted to prove the existence of a Lipschitz continuous Lyapunov function which is strict outside $\mathcal{SCR}(\phi)$. The proof --see Theorem \ref{main thm}-- is preceded by two technical lemmas aiming to overcome the fact that the function (\ref{definizione u_T}) constructed in Proposition \ref{proprieta u_T} is only definitively a strict Lyapunov function for $\phi$. We then proceed to discuss some consequences of Theorem \ref{main thm} and to give various examples. \\
\indent For any $T > 0$, we start by defining
\begin{equation}\label{definizione utilde}
\begin{split}
\tilde{u}_T&:X\rightarrow\R \\
\tilde{u}_T(x)&:= \max_{s\in[0,T]}u_T(\phi_s(x))
\end{split}
\end{equation}
and \begin{equation}\label{definizione baru}
\begin{split}
\bar{u}_T&:X\rightarrow\R \\
\bar{u}_T(x)&:=\frac{1}{M_T}\int_0^{+\infty}\dfrac{e^{-s}}{M_s}\tilde{u}_T(\phi_s(x)) ds.
\end{split}
\end{equation}
where the function $s \mapsto M_s$ has been defined in (\ref{misurabile}). We remark that the above definitions are inspired by the ones given by Conley in the proof of the so-called fundamental theorem of dynamical systems, see \cite{conl33}[Chapter II, Section 5, Page 33] and [Chapter II, Section 6, Page 39]. The main properties of functions (\ref{definizione utilde}) and (\ref{definizione baru}) are presented in the next two lemmas.
\begin{lemma}\label{proprieta utilde}
Let $\phi:X\times\R\rightarrow X$ be a continuous flow on a compact metric space $(X,d)$, uniformly Lipschitz continuous on compact subsets of $[0,+\infty)$. For any fixed $T > 0$, the following properties hold:
\begin{itemize}
\item[$(i)$] $\tilde{u}_T$ is $2M_T$-Lipschitz continuous.
\item[$(ii)$] $\tilde{u}_T$ is a Lyapunov function for $\phi$.
\item[$(iii)$] $\tilde{u}_T$ is definitively strict outside $\mathcal{A}_{\frac{T}{2}}$, that is
$$
\tilde{u}_T(\phi_t(x))<\tilde{u}_T(x) \qquad \qquad
\text{for any } x\in X \setminus \mathcal{A}_{\frac{T}{2}} \text{ and } t \ge T.$$
\end{itemize}
\end{lemma}
\begin{proof} Thanks to property $(i)$ of Proposition \ref{proprieta u_T} and assumption (\ref{flow ip}) on the flow, for any $x,y \in X$
$$
\tilde{u}_T(y)-\tilde{u}_T(x) = \max_{s\in[0,T]}u_T(\phi_s(y)) -\max_{s\in[0,T]}u_T(\phi_s(x))\leq \max_{s\in[0,T]}\left( u_T(\phi_s(y))-u_T(\phi_s(x)) \right)\leq
$$
$$
\leq 2\max_{s\in[0,T]} d(\phi_s(y),\phi_s(x)) \leq 2M_Td(x,y).
$$
We conclude by exchanging the role of $x$ and $y$. \\
$(ii)$ By property $(ii)$ of Proposition \ref{proprieta u_T}, we know that 
$$
u_T(\phi_{\tau}(x))\leq \max_{s\in[0,T]}u_T(\phi_s(x)) \qquad \qquad \forall x \in X \text{ and }\tau \ge 0.
$$
By taking the maximum for $\tau \in [t,T+t]$, we obtain
$$
\max_{\tau\in[t,T+t]}u_T(\phi_{\tau}(x))\leq\max_{s\in[0,T]}u_T(\phi_s(x)) = \tilde{u}_T(x)
$$
Since now
$$\max_{\tau\in[t,T+t]}u_T(\phi_{\tau}(x)) = \max_{s\in[0,T]}u_T(\phi_{s+t}(x)) = \tilde{u}_T(\phi_t(x)),$$
the thesis immediately follows. \\
$(iii)$ By property $(iii)$ of Proposition \ref{proprieta u_T}, for any $x\in X\setminus\mathcal{A}_{\frac{T}{2}}$ and $t \ge T$, we have
$$
\tilde{u}_T(\phi_t(x))=\max_{s\in[0,T]}u_T(\phi_{s+t}(x)) < u_T(x)\leq \max_{s\in[0,T]}u_T(\phi_s(x)) = \tilde{u}_T(x).
$$

\end{proof}
\begin{lemma}\label{proprieta baru}
Let $\phi:X\times\R\rightarrow X$ be a continuous flow on a compact metric space $(X,d)$, uniformly Lipschitz continuous on compact subsets of $[0,+\infty)$. For any fixed $T > 0$, the following properties hold:
\begin{itemize}
\item[$(i)$] $\bar{u}_T$ is $2$-Lipschitz continuous.
\item[$(ii)$] $\bar{u}_T$ is a Lyapunov function for $\phi$.
\item[$(iii)$] $\bar{u}_T$ is strict outside $\mathcal{A}_{\frac{T}{2}}$ for any $t > 0$.
\end{itemize}
\end{lemma}
\begin{proof} $(i)$ By property $(i)$ of Lemma \ref{proprieta utilde} and assumption (\ref{flow ip}) on the flow, we immediately obtain that for all $x,y\in X$
$$
\bar{u}_T(y)-\bar{u}_T(x) = \frac{1}{M_T}\int_0^{+\infty}\dfrac{e^{-s}}{M_s}\left( \tilde{u}_T(\phi_s(y))-\tilde{u}_T(\phi_s(x)) \right)ds\leq2\int_0^{+\infty}\dfrac{e^{-s}}{M_s} d(\phi_s(y),\phi_s(x))ds\leq 2 d(y,x).
$$
We conclude by exchanging the role of $x$ and $y$. \\
$(ii)$ The statement is a direct consequence of property $(ii)$ of Lemma \ref{proprieta utilde}. Indeed, for any $x \in X$ and $t\geq 0$, it clearly holds
$$
\bar{u}_T(\phi_t(x))-\bar{u}_T(x) = \frac{1}{M_T}\int_0^{+\infty}\dfrac{e^{-s}}{M_s}\left( \tilde{u}_T(\phi_{s+t}(x))-\tilde{u}_T(\phi_s(x)) \right)ds\leq 0
$$
$(iii)$ By definition,
$$
\bar{u}_T(\phi_t(x))-\bar{u}_T(x) = \frac{1}{M_T} \int_0^{+\infty} \dfrac{e^{-s}}{M_s}\left( \tilde{u}_T(\phi_{s+t}(x))-\tilde{u}_T(\phi_s(x)) \right) ds.
$$
\noindent Let us introduce
$$
\bar{s}:= \max\{ s\in[0,+,\infty]:\ \tilde{u}_T(\phi_s(x))=\tilde{u}_T(x) \}
$$
and notice that --by property $(iii)$ of Lemma \ref{proprieta utilde}-- $\bar{s}\in[0,T)$. Consequently, for any $x\in X\setminus\mathcal{A}_{\frac{T}{2}}$ and $t>0$, we have\begin{equation*}
\tilde{u}_T(\phi_{\bar{s}+t}(x))<\tilde{u}_T(x)=\tilde{u}_T(\phi_{\bar{s}}(x)).
\end{equation*}
By the previous strict inequality, we deduce that 
$$
\bar{u}_T(\phi_t(x))-\bar{u}_T(x)<0 \qquad \qquad \forall x \in x\in X\setminus\mathcal{A}_{\frac{T}{2}} \text{ and } t>0.
$$
\end{proof}
\indent We finally prove 
\begin{teorema}\label{main thm}
Let $\phi:X\times\R\rightarrow X$ be a continuous flow on a compact metric space $(X,d)$, uniformly Lipschitz continuous on compact subsets of $[0,+\infty)$. Then 
\begin{equation} \label{final}
\begin{split}
u&:X\rightarrow\R \\
x\mapsto u(x)&:= \sum_{n\in\N}\dfrac{1}{2^n}\bar{u}_n(x)
\end{split}
\end{equation}
is a Lipschitz continuous Lyapunov function for $\phi$ such that
$
\mathcal{N}(u) =\mathcal{SCR}(\phi).
$
\end{teorema}
\begin{proof} The Lipschitz continuity of $u$ is a direct consequence of property $(i)$ of Lemma \ref{proprieta baru}. Indeed, for any $x,y\in X$, we have
$$
u(y)-u(x)=\sum_{n\in\N}\dfrac{1}{2^n}\left( \bar{u}_n(y)-\bar{u}_n(x) \right)\leq 4d(x,y)
$$
and we conclude by exchanging the role of $x,y$.\\
\noindent Moreover, from property $(ii)$ of Lemma \ref{proprieta baru}, it holds
$$
u(\phi_t(x))-u(x)=\sum_{n\in\N}\dfrac{1}{2^n}\left(\bar{u}_n(\phi_t(x))-\bar{u}_n(x)\right)\leq 0\qquad \qquad \forall x \in X \text{ and } t \ge 0.$$
This means that $u$ is a Lyapunov function for $\phi$. \\
Finally, let $x \in X \setminus \mathcal{SCR}(\phi)$ and $t > 0$. We first notice that the strong chain recurrent set for $\phi$ --see (\ref{caratterizzazione scr at})-- can be equivalently  expressed as $\mathcal{SCR}(\phi) = \bigcap_{n \in \mathbb{N}}\mathcal{A}_{\frac{n}{2}}$. 
\noindent As a consequence, if $x\in X\setminus\mathcal{SCR}(\phi)= X\setminus\left( \bigcap_{n \in \mathbb{N}}\mathcal{A}_{\frac{n}{2}} \right)=\bigcup_{n \in \mathbb{N}}\left( X\setminus\mathcal{A}_{\frac{n}{2}} \right)$, then there exists $\bar{n}\in\N$ such that $x\in X\setminus\mathcal{A}_{\frac{\bar{n}}{2}}$. By property $(iii)$ of Lemma \ref{proprieta baru}, we then have
$$
\bar{u}_{\bar{n}}(\phi_t(x))-\bar{u}_{\bar{n}}(x)<0
$$
and therefore
$$u(\phi_t(x))-u(x) = \sum_{\substack{n\in\N \\ n\neq\bar{n}}}\dfrac{1}{2^n}\left(\bar{u}_n(\phi_t(x))-\bar{u}_n(x)\right) + \dfrac{1}{2^{\bar{n}}}\left(\bar{u}_{\bar{n}}(\phi_t(x))-\bar{u}_{\bar{n}}(x)\right)<0.
$$
This proves that $u$ is strict outside $\mathcal{SCR}(\phi)$, that is $\mathcal{N}(u) \subseteq \mathcal{SCR}(\phi)$. However, by Proposition 1.6 in \cite{abbbercar}, the neutral set of every Lipschitz continuous Lyapunov function for $\phi$ contains the strong chain recurrent set of $\phi$, so $\mathcal{SCR}(\phi) \subseteq \mathcal{N}(u)$. We then conclude that $
\mathcal{N}(u) =\mathcal{SCR}(\phi)$. 
\end{proof}
\indent The above theorem implies the following characterization of the strong chain recurrent set in terms of Lipschitz continuous Lyapunov functions. 
\begin{corollario} \label{altra dim}
Let $\phi:X\times\R\rightarrow X$ be a continuous flow on a compact metric space $(X,d)$, uniformly Lipschitz continuous on the compact subsets of $[0,+\infty)$. Then, 
\begin{equation*}
\mathcal{SCR}(\phi)=\bigcap_{f\in\mathcal{L}(\phi)}\mathcal{N}(f).
\end{equation*}
where $\mathcal{L}(\phi)$ denotes the set of all Lipschitz continuous Lyapunov functions for $\phi$. 
\end{corollario}
\begin{proof} By Proposition 1.6 in \cite{abbbercar}, the neutral set of every Lipschitz continuous Lyapunov function for $\phi$ contains the strong chain recurrent set of $\phi$, that is 
\begin{equation*}
\mathcal{SCR}(\phi) \subseteq \bigcap_{f\in\mathcal{L}(\phi)}\mathcal{N}(f).
\end{equation*}
Since --by Theorem \ref{main thm}-- the function $u$ defined in (\ref{final}) is such that $\mathcal{SCR}(\phi) = \mathcal{N}(u)$, the thesis immediately follows.
\end{proof}
\noindent We stress that the corollary above has been already proved in \cite{abbbercar}[Theorem 2.2]. However, in \cite{abbbercar} the inclusion $\supseteq$ is obtained by an accurate adaptation of Fathi and Pageault's techniques. While, in our proof, it is an immediate consequence of the existence of the Lyapunov function (\ref{final}). \\
~\newline
\indent Let now $(M,g)$ be a $C^{\infty}$ closed connected Riemannian manifold and $V: M \to TM$ be a $C^k$ vector field, $k \ge 2$, inducing the flow $\phi: \R \times M \to M$. Denote by $\| v \|_x$ the norm of an element $v \in T_xM$ relatively to the metric $g$ and introduce the Mañé Hamiltonian
$$H_V: T^*M \to \R, \qquad H_V(x,p) = \frac{1}{2} \| p \|^2_x + p(V(x)).$$
We indicate by $\mathcal{A}_V$ the projected Aubry set associated to $H_V$. Since the constant functions are solutions of the Hamilton-Jacobi equation $H_V(x,d_xv) = 0$, the Mañé critical level is $c(H_V) = 0$. From one hand, it holds that $\mathcal{A}_V \subseteq \mathcal{CR}(\phi)$, see e.g. \cite{mane}[Section I, Theorem V], \cite{pageault}[Introduction, Section 3, Proposition 3.2] and {\cite{bernardconley}[Corollary 2]}. From the other hand, we know that there exists a $C^{1,1}$ critical sub-solution $u:M \to \R$ of $H_V(x,d_xv) = 0$ which is strict outside $\mathcal{A}_V$,
see \cite{bernardexistence}[Lemma 7]. In particular, $d_xu(V(x)) \le 0$ and therefore $u$ is a $C^{1,1}$ Lyapunov function for $\phi$. Since $u$ is strict outside $\mathcal{A}_{V}$, it holds $\mathcal{N}(u)\subseteq\mathcal{A}_V$. Furthermore, recall that for any $x\in\mathcal{A}_V$ and any sub-solution {\color{red}{$v$}}
\begin{equation}
\frac{1}{2} \| d_xv \|^2_x + d_xv(V(x)) \le 0,
\end{equation}
$d_xv$ does not depend on $v$ (see \cite{alfonso}[Proposition 5.1.23]). Consequently, since the constant functions are (sub-)solutions of $H(x,d_xv)=0$, we have $d_xv=0$ on $\mathcal{A}_V$. In particular, this implies $d_xu(V(x))=0$ for all $x\in\mathcal{A}_V$ or equivalently $\mathcal{A}_V\subseteq\mathcal{N}(u)$. Summarizing, $u$ is a $\mathcal{C}^{1,1}$ Lyapunov function for $\phi$ with $\mathcal{N}(u)=\mathcal{A}_V$. Therefore, by the previous corollary, $\mathcal{SCR}(\phi) \subseteq \mathcal{A}_V$. From both inclusions, we conclude that in general
\begin{equation}\label{relazione Mane generale}
\mathcal{SCR}(\phi) \subseteq \mathcal{A}_V \subseteq \mathcal{CR}(\phi).
\end{equation}

\begin{corollario}\label{corollario mane}

If $\mathcal{SCR}(\phi)=\mathcal{CR}(\phi)$ then there exists a ${\cal{C}}^{1,1}$ Lyapunov function such that $
\mathcal{N}(u) =\mathcal{SCR}(\phi).
$
\end{corollario}
\begin{proof} If $\mathcal{SCR}(\phi) = \mathcal{CR}(\phi)$ then necessarily $\mathcal{SCR}(\phi) = \mathcal{A}_V = \mathcal{CR}(\phi)$ and the $C^{1,1}$ critical sub-solution $u:M \to \R$ of $H_V(x,d_xu) = 0$ which is strict outside $\mathcal{A}_V$ gives the desired Lyapunov function with $
\mathcal{N}(u) =\mathcal{SCR}(\phi).$
\end{proof}
\indent Next basic Example \ref{41} shows two different smooth flows on $\mathbb{T}^1 = \R/\Z$, having the following properties: 
\begin{itemize}
\item[$(i)$]  For the first flow, the hypothesis of Corollary \ref{corollario mane} holds and therefore it exists a differentiable Lyapunov function which is strict outside the strong chain recurrent set. 
\item[$(ii)$] For the second flow, we have that $\mathcal{SCR}(\phi) \subsetneqq \mathcal{A}_V = \mathbb{T}^1$. In such a case, even if every sub-solution given by weak KAM theory provides a first integral for $\phi$, the dynamical system admits a differentiable Lyapunov function strict outside $\mathcal{SCR}(\phi)$.
\end{itemize}
\begin{esempio} \label{41} On the circle $\mathbb{T}^1 = \R/\Z$ equipped with the standard quotient metric, consider the dynamical systems of Figures \ref{figura1} and \ref{figura2}, where the bold line and the arrows denote respectively fixed points and the direction of the flow $\phi$. \\
Relatively to the dynamical system of Figure \ref{figura1}, the hypothesis of Corollary \ref{corollario mane} is satisfied since $\mathcal{SCR}(\phi) = \mathcal{CR}(\phi) = Fix(\phi)$ and it clearly exists a $C^{1,1}$ (also $C^{\infty}$) Lyapunov function whose neutral set is the strong chain recurrent set. \\
Otherwise, in the case of the dynamical system of Figure \ref{figura2}, $\mathcal{SCR}(\phi) = Fix(\phi)$ and $\mathcal{CR}(\phi) = \mathbb{T}^1$. Moreover --see point $(i)$ of Theorem 4.8 and point $(i)$ of Lemma 4.14 in \cite{fathifigallirifford}--  the projected Aubry set associated to the corresponding Mañé Hamiltonian coincides with $\mathcal{CR}(\phi) = \mathbb{T}^1$. Consequently, every sub-solutions given by weak KAM theory is a first integral but --according to Theorem \ref{main thm}-- it clearly exists a Lipschitz continuous (even smooth) Lyapunov function strict outside $Fix(\phi)$. 
\end{esempio}

\begin{figure}[h]
\centering
\begin{minipage}[c]{.3\textwidth}
\centering
\begin{tikzpicture}
\draw (3,3) circle (15mm);
\draw [ultra thick] (3,4.5) arc (90:180:1.5);
\draw [ultra thick] (3,1.5) arc (-90:0:1.5);
\draw [-latex] (3,4.5) arc (90:45:1.5);
\draw [-latex] (1.5,3) arc (-180:-135:1.5);
\end{tikzpicture}
\caption{First dynamical system of Example \ref{41}.}
\label{figura1}
\end{minipage}
\hspace{75pt}
\begin{minipage}[c]{.3\textwidth}
\centering
\begin{tikzpicture}
\draw (3,3) circle (15mm);
\draw [ultra thick] (3,4.5) arc (90:180:1.5);
\draw [ultra thick] (3,1.5) arc (-90:0:1.5);
\draw [-latex] (3,4.5) arc (90:45:1.5);
\draw [-latex] (3,1.5) arc (-90:-135:1.5);
\end{tikzpicture}
\caption{Second dynamical system of Example \ref{41}.}
\label{figura2}
\end{minipage}
\end{figure}

\noindent Another case where sub-solutions of weak KAM theory provide first integrals while the Lyapunov function of Theorem \ref{main thm} is not a first integral, is discussed in the next 
\begin{esempio}\label{es2} 
On $\mathbb{T}^2=\R^2/\Z^2$ endowed with the standard quotient metric, consider the flow $\phi$ associated to the vector field (see Figure \ref{figuraMM}):
\begin{equation} \label{arnaud}
V(x,y) = (f(x),1)
\end{equation}
where
\begin{equation*}
f(x) =
\begin{cases}
\cos(4\pi x) + 1 & \qquad x \in \left[0,\frac{1}{4}\right] \cup \left[\frac{3}{4},1\right) \\
0 & \qquad x \in \left[\frac{1}{4},\frac{3}{4}\right]
\end{cases}
\end{equation*}
\noindent In such a case, the strong chain recurrent set coincides with the set of periodic points, that is
$$\mathcal{SCR}(\phi) = \left[\frac{1}{4},\frac{3}{4}\right]\times\mathbb{T}^1$$
and $\mathcal{CR}(\phi) = \mathbb{T}^2$. Moreover --by points $(i)$ of Theorem 4.8 and $(i)$ of Lemma 4.14 in \cite{fathifigallirifford} again-- the projected Aubry set associated to the Mañé Hamiltonian $$\frac{1}{2} (p_1^2 + p_2^2) + f(x)p_1 + p_2$$ is $\mathcal{A}_V=\mathcal{CR}(\phi) = \mathbb{T}^2$. Consequently, every sub-solutions given by weak KAM theory is a first integral for $\phi$ while function (\ref{final}) of Theorem \ref{main thm} gives a Lyapunov function for $\phi$ which is strict outside $\left[\frac{1}{4},\frac{3}{4}\right]\times\mathbb{T}^1$. 
\begin{figure}[h]
\centering
\includegraphics[scale=.50]{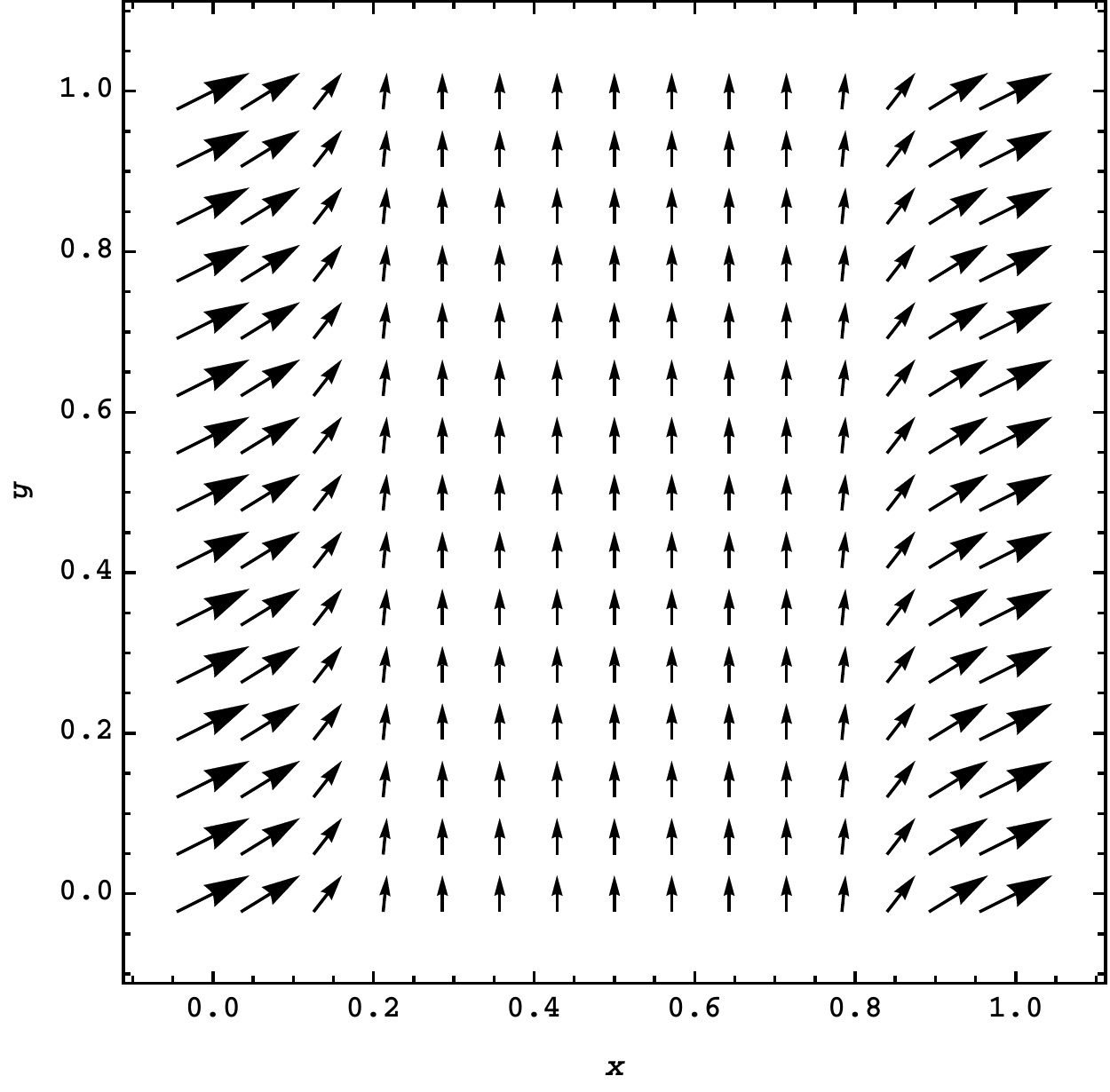}
\caption{The vector field (\ref{arnaud}).}
\label{figuraMM}
\end{figure}
\end{esempio}
\indent We finally remark that the converse implication of Corollary \ref{corollario mane} does not in general hold true. For example, consider the dynamical system of Figure \ref{figura2}. In such a case, it is easy to construct a smooth Lyapunov function which is strict outside $\mathcal{SCR}(\phi) = Fix(\phi)$. However, $\mathcal{SCR}(\phi) \subsetneqq \mathcal{CR}(\phi) = \mathbb{T}^1$. A necessary and sufficient condition for the strong chain recurrent and chain recurrent sets to be equal is proved in Proposition \ref{ultima car} of the next section. 


\section{Strong chain transitivity and Lipschitz Lyapunov functions} \label{sei}
In this section, $\phi:X\times\R\rightarrow X$ is a continuous flow on a compact metric space $(X,d)$. \\
\indent We first introduce $L_T$-dominated and $L$-dominated functions and then we explain their relation with Lipschitz continuous Lyapunov functions and the strong chain recurrent set. Next definitions and results are the continuous-time versions of the ones appearing in \cite{pageault}[Chapter 2, Section 4] and \cite{fathi2015}[Section 2.2] for homeomorphisms. 
\begin{definizione} \label{domi} \emph{(Dominated function)}
\begin{itemize}
\item[$(i)$] A function $f:X\rightarrow\R$ is said to be $L_T$-dominated if there exists $K_T>0$ such that 
\begin{equation*}
f(y)-f(x)\leq K_T L_T(x,y) \qquad \qquad \forall x,y \in X.
\end{equation*}
\item[$(ii)$] A function $f:X\rightarrow\R$ is said to be $L$-dominated if it is $L_T$-dominated for any $T > 0$.
\end{itemize}
\end{definizione}
\noindent The proof of the next proposition is an easy adaptation for flows of the proof of Lemma 4.3 in \cite{pageault}.
\begin{proposizione} \label{dom-def}
A Lipschitz continuous function $f:X\rightarrow\R$ is $L_T$-dominated if and only if it is definitively a Lyapunov function for $\phi$, that is
$$
f(\phi_t(x))\leq f(x) \qquad \qquad \text{for any } x\in X \text{ and } t\geq T.$$
\end{proposizione}
\begin{proof}
From one hand, let $f:X\rightarrow\R$ be a $L_T$-dominated, Lipschitz continuous function for $\phi$. Thanks to point $(ii)$ of Proposition \ref{proprieta L_T}, for any fixed $x\in X$ and $t\geq T$ it holds
$$
f(\phi_t(x))-f(x)\leq K_TL_T(x,\phi_t(x))=0,
$$
that is $f$ is a definitively Lyapunov function.\\
\noindent On the other hand, let $f:X\rightarrow\R$ be a $K$-Lipschitz continuous, definitively  Lyapunov function for $\phi$. Fixed $x,y\in X$, let $C = (x_i,t_i)_{i =1,\ldots,n} \in C_T(x,y)$ be a chain from $x$ to $y$. Then
$$
f(x_{i+1})-f(x_i)\leq f(x_{i+1})-f(\phi_{t_i}(x_i))\leq Kd(\phi_{t_i}(x_i),x_{i+1}) \qquad \qquad \forall i=1,\dots,n.
$$
\noindent By adding all the inequalities, we conclude that
$$
f(y)-f(x) = \sum_{i=1}^n \left( f(x_{i+1})-f(x_i) \right)\leq K\sum_{i=1}^n d(\phi_{t_i}(x_i),x_{i+1}) = K \, l_T(C).
$$
More precisely, by considering the infimum over all possible chains in $C_T(x,y)$ we have
\begin{equation}\label{precisely}
f(y)-f(x)\leq KL_T(x,y)\qquad\forall x,y\in X
\end{equation}
where $K$ is any Lipschitz constant for $f$.
\end{proof}
\begin{corollario}
A Lipschitz continuous function $f:X\rightarrow\R$ is $L$-dominated if and only if it is a Lyapunov function for $\phi$.
\end{corollario}

\noindent In the sequel, we denote by $\mathcal{L}^1_T(\phi)$ the set of $1$-Lipschitz continuous functions $f:X\to \mathbb{R}$ such that $f(\phi_t(x)) \le f(x)$ for any $x\in X$ and $t \ge T$. The next lemma is the continuous-time version of Proposition 4.6 in \cite{pageault}.
\begin{lemma} \label{lemma 11} For any $x \in \mathcal{SCR}(\phi)$ and $y \in X$, we have
\begin{equation} \label{elle 1}
L_T(x,y) = \sup_{f \in \mathcal{L}^1_T(\phi)} f(y) - f(x).
\end{equation}
\end{lemma}
\begin{proof}  By the previous Proposition \ref{dom-def}, if $f \in \mathcal{L}^1_T(\phi)$ then $f$ is $L_T$-dominated; moreover, see inequality (\ref{precisely}), it results that $K_T = 1$. This means that
$$f(w) - f(z) \le L_T(z,w) \qquad \forall z,w \in X$$
and consequently
$$L_T(z,w) \ge \sup_{f \in \mathcal{L}^1_T(\phi)} \left( f(w) - f(z) \right) \qquad \forall z,w \in X.$$
In order to prove the other inequality, let us fix $x \in \mathcal{SCR}(\phi)$ and define $f_{x}(\cdot) := L_T(x,\cdot)$. We notice that the function $f_x(\cdot) \in \mathcal{L}^1_T(\phi)$. Indeed:
$$f_x(z) - f_x(y) = L_T(x,z) - L_T(x,y) \le d(y,z) \qquad \forall y,z \in X$$
and 
$$f_x(\phi_t(y)) - f_x(y) = L_T(x,\phi_t(y)) - L_T(x,y) \le L_T(y,\phi_t(y)) = 0$$
for any $y \in X$ and $t \ge T$. Moreover, it holds that
$$L_T(x,y) = L_T(x,y) - L_T(x,x) = f_x(y) - f_x(x) $$
because $L_T(x,x) = 0$. As a consequence,
$$L_T(x,y) \le \sup_{f \in \mathcal{L}^1_T(\phi)} \left( f(y) - f(x) \right) \qquad \forall x \in \mathcal{SCR}(\phi) \text{ and } y \in X.$$
\end{proof} 
\indent Previous Lemma \ref{lemma 11} leads to prove the following results for the set $\mathcal{L}(\phi)$ of all Lipschitz continuous Lyapunov functions for $\phi$. The corresponding version for homeomorphisms is given by Corollary 4.7 and Theorem 4.8. of Chapter 2 in \cite{pageault}.
\begin{proposizione} \label{15 maggio}
Any $f\in\mathcal{L}(\phi)$ is constant on every strong chain transitive component of $\mathcal{SCR}(\phi)$. Moreover, if $x,y\in\mathcal{SCR}(\phi)$ belong to different strong chain transitive components of $\mathcal{SCR}(\phi)$, then there exists a function $f\in\mathcal{L}(\phi)$ such that $f(x)\neq f(y)$.
\end{proposizione}
\begin{proof} Given $f \in \mathcal{L}(\phi)$ with Lipschitz constant $K_f > 0$, it clearly holds that $\frac{f}{K_f} \in \mathcal{L}^1_T(\phi)$ for any $T > 0$. Let now $x,y \in \mathcal{SCR}(\phi)$ be in the same strong chain transitive component of $\mathcal{SCR}(\phi)$, that is $L_T(x,y) = 0 = L_T(y,x)$ for all $T > 0$. By equality (\ref{elle 1}), one has
$$
0=L_T(x,y) = \sup_{g\in\mathcal{L}^1_T(\phi)}g(y)-g(x)\geq \dfrac{f(y)}{K_f}-\dfrac{f(x)}{K_f}
$$
and
$$
0=L_T(y,x) = \sup_{g\in\mathcal{L}^1_T(\phi)}g(x)-g(y)\geq \dfrac{f(x)}{K_f}-\dfrac{f(y)}{K_f}
$$
so that $f(y) = f(x)$. \\
Conversely, let $x,y \in \mathcal{SCR}(\phi)$ be in different strong chain transitive components of $\mathcal{SCR}(\phi)$. Exchanging the role of $x,y$ if needed, this means that there exists a time $T > 0$ such that $L_T(x,y) > 0$. By using equality (\ref{elle 1}) again, we conclude that there exists a function $f \in \mathcal{L}^1_T(\phi)$ such that $f(y) \ne f(x)$. Let us now define 
$$\tilde{f}(z) := \max_{s \in [0,T]} f(\phi_s(z))$$
which is in $\mathcal{L}(\phi)$. Moreover, there exist $\bar{s} = \bar{s}(x) \in [0,T]$ and $\hat{s} = \hat{s}(y) \in [0,T]$ such that 
$$\tilde{f}(x) = f(\phi_{\bar{s}}(x)) \qquad \text{and} \qquad \tilde{f}(y) = f(\phi_{\hat{s}}(y)).$$
Since now --by Proposition \ref{invarianza}-- if $z \in \mathcal{SCR}(\phi)$ and $s \in \R$ then $z$ and $\phi_s(z)$ belong to the same strong chain transitive component of $\mathcal{SCR}(\phi)$, we conclude that
$$\tilde{f}(x) = f(\phi_{\bar{s}}(x)) = f(x) \qquad \text{and} \qquad \tilde{f}(y) = f(\phi_{\hat{s}}(y)) = f(y).$$
As a consequence, we have proved that $\tilde{f} \in \mathcal{L}(\phi)$ is such that $\tilde{f}(x)\neq \tilde{f}(y)$.
\end{proof}
\noindent The proof of the next theorem is essentially the same of Theorem 4.8 in Chapter 2 in \cite{pageault} and it is omitted. We notice that Fathi and Pageault's statement is formulated by using the so-called $d$-Mather classes, corresponding exactly to the strong chain transitive components of the strong chain recurrent set for a given homeomorphism.
\begin{teorema} \label{FP for flows}
Let $\phi:X\times\R\rightarrow X$ be a continuous flow on a compact metric space $(X,d)$. $\mathcal{SCR}(\phi)$ has a unique strong chain transitive component if and only if the only Lipschitz continuous Lyapunov functions for $\phi$ are the constants.
\end{teorema}
\indent In the fundamental theorem of dynamical systems recalled in the introduction (see \cite{conl33}[Chapter II]), Conley made explicit a so-called complete Lyapunov function for $\phi$, whose properties are listed below.
\begin{definizione} \label{completa} \emph{(Complete Lyapunov function)} \\
Let $\phi:X\times\R\rightarrow X$ be a continuous flow on a compact metric space $(X,d)$. A continuous Lyapunov function $f:X\rightarrow\R$ for $\phi$ is called complete if:
\begin{itemize}
\item[$(i)$] $\mathcal{N}(f) = \mathcal{CR}(\phi)$.
\item[$(ii)$] If $x,y \in \mathcal{CR}(\phi)$, then $f(x) = f(y)$ if and only if $x$ and $y$ belong to the same chain transitive component of $\mathcal{CR}(\phi)$.
\item[$(iii)$] $f(\mathcal{CR}(\phi))$ is a compact nowhere dense subset of $\R$. 
\end{itemize}
\end{definizione}
\noindent We notice that the function given by Theorem \ref{main thm} cannot be in general assumed to be complete with respect to $\mathcal{SCR}(\phi)$. For example, look at the dynamical system of Figure \ref{figura2}. In such a case, every fixed point is a strong chain transitive component of $\mathcal{SCR}(\phi) = Fix(\phi)$ but it is not possible to construct a continuous Lyapunov function such that properties:
\begin{itemize}
\item[$(i)'$] $\mathcal{N}(f) = \mathcal{SCR}(\phi)$
\end{itemize}
and
\begin{itemize}
\item[$(ii)'$] If $x,y \in \mathcal{SCR}(\phi)$, then $f(x) = f(y)$ if and only if $x$ and $y$ belong to the same strong chain transitive component of $\mathcal{SCR}(\phi)$
\end{itemize}
simultaneously hold. In particular, if we change the setting from chain recurrence to strong chain recurrence, the next notion of pseudo-complete Lyapunov function is crucial. This definition was first introduced (at least at our knowledge) by Katsuya Yokoi for homeomorphisms (see \cite{yok}[Definition 5.1] and also \cite{wis}[Section 6]).
\begin{definizione} \emph{(Pseudo-complete Lyapunov function)} \label{pseudocompleta} \\
Let $\phi:X\times\R\rightarrow X$ be a continuous flow on a compact metric space $(X,d)$.  A continuous Lyapunov function $f:X\rightarrow\R$ for $\phi$ is called pseudo-complete if $\mathcal{N}(f) = \mathcal{SCR}(\phi)$ and $f$ is constant on every strong chain transitive component of $\mathcal{SCR}(\phi)$. 
\end{definizione}
\noindent Indeed, from Theorem \ref{main thm} and the first part of Proposition \ref{15 maggio}, we immediately deduce the next 
\begin{corollario} Let $\phi:X\times\R\rightarrow X$ be a continuous flow on a compact metric space $(X,d)$, uniformly Lipschitz continuous on compact subsets of $[0,+\infty)$. Then there exists a pseudo-complete Lyapunov function for $\phi$. 
\end{corollario} 
\noindent We finally remark that the notion of pseudo-complete Lyapunov function is useful to discuss the condition for the strong chain recurrent and chain recurrent sets to be equal (see \cite{yok}[Theorem 5.3] for the case of a homeomorphism).
\begin{proposizione} \label{ultima car} Let $\phi:X\times\R\rightarrow X$ be a continuous flow on a compact metric space $(X,d)$. Then, $\mathcal{SCR}(\phi) = \mathcal{CR}(\phi)$ if and only if there exists a pseudo-complete Lyapunov function $u: X \to \R$ for $\phi$ such that $u(\mathcal{SCR}(\phi))$ is totally disconnected. 
\end{proposizione}
\begin{proof} Let $\mathcal{SCR}(\phi)=\mathcal{CR}(\phi)$. By Conley's fundamental theorem (see \cite{conl33}[Chapter II, Section 6.4]), there exists a complete Lyapunov function $u$ for $\phi$. Clearly, $u$ is also a pseudo-complete Lyapunov function for $\phi$ such that $u(\mathcal{SCR}(\phi))=u(\mathcal{CR}(\phi))$ is totally disconnected.\\
\noindent Conversely, let $u:X\rightarrow\R$ be a pseudo-complete Lyapunov function for $\phi$ such that $u(\mathcal{SCR}(\phi))$ is totally disconnected. Without loss of generality, we suppose that $u(x) \ge 0$ for any $x \in X$. In order to show that $\mathcal{CR}(\phi)\subseteq\mathcal{SCR}(\phi)$, we remind that the next conditions are equivalent:
\begin{itemize}
\item[$(i)$] $x\notin\mathcal{CR}(\phi)$.
\item[$(ii)$] There exists an attractor $K$ such that $x\notin K$ but $\omega(x)\subseteq K$.
\end{itemize}
See \cite{conl33}[Chapter II, Section 6.2, Page 37] for this equivalence. Let $x\notin \mathcal{SCR}(\phi)$ and $\bar{t}>0$ such that $u(\phi_{\bar{t}}(x))<u(x)$. Since $u(\mathcal{SCR}(\phi))$ is totally disconnected, there exists $r_0 > 0$ such that $u(\phi_{\bar{t}}(x))<r_0<u(x)$ and $r_0\notin u(\mathcal{SCR}(\phi))$. \noindent We define $U:=\{y\in X:\ u(y)\in[0,r_0)\}$. Since $u$ is a Lyapunov function for $\phi$, we have that $K:=\omega(U)\subseteq U$, which means that $K$ is an attractor. Moreover, $x\notin K$ but, since $\phi_{\bar{t}}(x)\in U$, $\omega(x)\subseteq K$. From the equivalence of points $(i)$ and $(ii)$ recalled above, we conclude that $x\notin\mathcal{CR}(\phi)$. This prove that $\mathcal{CR}(\phi)\subseteq\mathcal{SCR}(\phi)$ and therefore $\mathcal{CR}(\phi) = \mathcal{SCR}(\phi)$.
\end{proof}

\addcontentsline{toc}{chapter}{Bibliography}
\bibliographystyle{siam}
\bibliography{BiblioArticolo2}

\end{document}